\documentclass[12pt,oneside]{amsart}
\usepackage{amssymb,verbatim}

\numberwithin{equation}{section}
\newtheorem{theorem}{Theorem}[section]
\newtheorem{lemma}[theorem]{Lemma}

\newtheorem{corollary}[theorem]{Corollary}

\theoremstyle{definition}
\newtheorem{definition}[theorem]{Definition}
\newtheorem{remark}[theorem]{Remark}
\newtheorem{example}[theorem]{Example}

\newtheorem{notation}[theorem]{Notation}
\newtheorem{question}[theorem]{Question}

\DeclareMathOperator{\Ann}{Ann}
\DeclareMathOperator{\Ass}{Ass}

\newcommand{\K}{\mathcal{K}}

\newcommand{\N}{\mathbb{N}}

\newcommand{\mf}{\mathfrak{m}}

\newcommand{\HH}{\mathcal{H}}

\newcommand{\E}{\mathcal{E}}
\newcommand{\bm}{{\bf m}}

\begin{document}


\title[Hypergraphs and associated primes]{
Colorings of hypergraphs, perfect graphs,
and associated primes of powers of monomial ideals}

\author{Christopher A. Francisco}
\address{Department of Mathematics, Oklahoma State University,
401 Mathematical Sciences, Stillwater, OK 74078}
\email{chris@math.okstate.edu}
\urladdr{http://www.math.okstate.edu/$\sim$chris}

\author{Huy T\`ai H\`a}
\address{Tulane University \\ Department of Mathematics \\
6823 St. Charles Ave. \\ New Orleans, LA 70118, USA}
\email{tai@math.tulane.edu}
\urladdr{http://www.math.tulane.edu/$\sim$tai/}

\author{Adam Van Tuyl}
\address{Department of Mathematical Sciences \\
Lakehead University \\
Thunder Bay, ON P7B 5E1, Canada}
\email{avantuyl@lakeheadu.ca}
\urladdr{http://flash.lakeheadu.ca/$\sim$avantuyl/}

\thanks{Version: \today}

\begin{abstract}
There is a natural one-to-one correspondence between squarefree monomial ideals and finite simple hypergraphs via the cover ideal construction. Let $\HH$ be a finite simple hypergraph, and let $J = J(\HH)$ be its cover ideal in a polynomial ring $R$. We give an explicit description of all associated primes of $R/J^s$, for any power $J^s$ of $J$, in terms of the coloring properties of hypergraphs arising from $\HH$. We also give an algebraic method for determining the chromatic number of $\HH$, proving that
it is equivalent to a monomial ideal membership problem involving powers of $J$.
Our work yields two new purely algebraic characterizations of perfect graphs, 
independent of the Strong Perfect Graph Theorem; the first characterization is in terms 
of the sets $\Ass(R/J^s)$, while the second characterization 
is in terms of the saturated chain condition for associated primes.
\end{abstract}

\maketitle


\section{Introduction}

Our goal in this paper is to investigate an intimate connection between associated primes of powers of squarefree monomial ideals and the coloring properties of finite simple hypergraphs.

Throughout this paper, $\HH$ is a finite simple hypergraph on vertices $V_{\HH}=\{x_1,\dots,x_n\}$
with edge set $\mathcal{E}_{\HH}=\{E_1,\dots,E_t\}$, where the $E_i$ are subsets of
$\{x_1,\dots,x_n\}$ of cardinality at least two (so $\HH$ has no loops), and
$E_i \not \subseteq E_j$ for $i \not = j$ (no multiple edges). When $|E_i|=2$ for all $i$,
we have a finite simple graph, and we shall frequently specialize to this case to prove
results about graphs. Let $k$ be a field, and identify the vertices of $\HH$ with the
variables in a polynomial ring $R = k[x_1, \dots, x_n]$. Two primary notions to connect
the combinatorics of a hypergraph $\HH$ with commutative algebra are the \textbf{edge ideal}
\[I(\HH) =
(x_{i_1} \cdots x_{i_r} ~\big|~ \{x_{i_1}, \dots, x_{i_r}\} \in \mathcal{E}_{\HH}) \subseteq R, \]
and the \textbf{cover ideal}
\[ J(\HH) =
(x_{j_1} \cdots x_{j_l} ~\big|~ \{x_{j_1}, \dots, x_{j_l}\} \text{ is a vertex cover of } \HH)
\subseteq R.\]
The notion of an edge ideal was first introduced for graphs in \cite{V1} and extended to
hypergraphs in \cite{HVT2}. The cover ideal $J(\HH)$ can be realized as the squarefree
Alexander dual $I(\HH)^\vee$ of $I(\HH)$. Both the edge ideal and the cover ideal constructions give one-to-one correspondences between squarefree monomial ideals and finite simple hypergraphs. In this paper, we shall concentrate on the cover ideal $J(\HH)$ and its powers; and thus, complementing \cite{HHT,Hoa,V2}, we study squarefree monomial ideals and their powers.

Our work was initially inspired by attempts to understand perfect graphs from an algebraic
perspective. The Strong Perfect Graph Conjecture, formulated by Berge over 40 years ago, asserted
that a graph $G$ is perfect if and only if neither $G$ nor its complement contains an
induced cycle of odd length at least five. The Strong Perfect Graph Theorem (SPGT), which verifies
Berge's conjecture, was recently proven in 
\cite{CRST}. In \cite{FHVT}, we used the SPGT to give an
algebraic method of determining whether a graph $G$ is perfect. 
Here, we look for characterizations of perfect graphs {\it without} using the SPGT,
perhaps yielding an approach that could lead to an algebraic proof of the
SPGT.

Although our motivation comes from a problem in graph theory, the scope of our work extends well beyond our initial purpose. We investigate a number of questions in commutative algebra and combinatorics, including:

\begin{enumerate}
\item (Computing the chromatic number) Is there an algebraic method for computing
$\chi(\HH)$, the chromatic number of $\HH$, for a given hypergraph $\HH$?
\item (Describing associated primes) What are the associated primes of powers of a squarefree monomial ideal? How are the associated primes of the powers of $J(\HH)$
reflected in the coloring properties of $\HH$ and its subhypergraphs?
\item (Stability of associated primes) What are bounds on an integer $a$, for a given
ideal $I$, such that $\bigcup_{s=1}^\infty \Ass(R/I^s)$ stabilizes at $a$, i.e.,
\[ \bigcup_{s=1}^a \Ass(R/I^s) = \bigcup_{s=1}^\infty \Ass(R/I^s)?\]
\item (Characterizing perfect graphs) How can one give an algebraic characterization of perfect
graphs without using the Strong Perfect Graph Theorem?
\item (Persistence of associated primes) Given any ideal $I$, what conditions on $I$
ensure that $\Ass(R/I^s) \subseteq \Ass(R/I^{s+1})$ for all $s \ge 1$?
\item (Saturated chain property for associated primes) An ideal $I \subset R$ has the \textbf{saturated chain
property for associated primes} if given any associated prime $P$ of $R/I$ that is not
minimal, there exists an associated prime $Q \subsetneq P$ with height$(Q)=$ height$(P)-1$.
Identify families of ideals possessing this (relatively rare) property.
\end{enumerate}

Our first main theorem, which answers (1),
is a mechanism for determining $\chi(\HH)$.

\begin{theorem} [Theorem~\ref{colortheorem}]
$\chi(\HH)$ is the
minimal $d$ such that $(x_1\cdots x_n)^{d-1} \in J(\HH)^d$.
\end{theorem}
\noindent
An alternative algebraic description of $\chi(\HH)$ was first given by Sturmfels and Sullivant \cite{SS} 
in terms of the $s$-th secant ideal of $I(\HH)$.  When $\HH = G$ is a graph,
Theorem 1.1 can be generalized to compute the $b$-fold chromatic number (see Theorem
\ref{bfold}) of $G$, a number which is related to the fractional chromatic number.

An answer to question (2) is much more subtle.  Roughly speaking,
we will show that the induced subhypergraphs that are critically $(s+1)$-chromatic,
that is, their chromatic number is $s+1$, but any proper induced
subhypergraph has a smaller chromatic number, each contribute an
element to $\Ass(R/J^s)$.
The subtlety to question (2) comes from the
fact that one needs to look for critically $(s+1)$-chromatic induced
subhypergraphs not in $\HH$, but in the {\bf $s$-th expansion} of $\HH$ (see
Definition \ref{expansion}).  By generalizing results of
Sturmfels and Sullivant \cite{SS} and using the generalized Alexander dual $(J^s)^{[\mathbf{s}]}$ of $J^s$, 
where $[\mathbf{s}] = (s, \dots, s)$, our next main theorem and its corollary (Corollary \ref{classificationofass}) give a complete answer to question (2).

\begin{theorem} [Theorem~\ref{assoc-depolar}]
Let $\HH = (V_\HH, \mathcal{E}_\HH)$ be a finite simple hypergraph 
with cover ideal $J = J(\HH)$.
For any $s \ge 1$, we have
\[(J^s)^{[\mathbf{s}]} = (\overline{\bm_T} ~\big|~ \chi(\HH^s_T) > s )\]
where $\HH^s_T$ is the induced subgraph of $\HH^s$, the $s$-th expansion of $\HH$, over a collection of its vertices $T$, and $\overline{\bm_T}$ is the depolarization of $\bm_T$.
\end{theorem}

A similar construction to the
expansion of $\HH$, the parallelization of $\HH$, appears in recent work
of Mart\'{i}nez-Bernal, Renter\'{i}a, and Villarreal \cite{MRV}. 
These constructions suggest that
when studying powers of edge or cover ideals of $\HH$, information
about these ideals is not only encoded into $\HH$, but also into
the expansion or parallelization of $\HH$.

Question (3) was inspired by results of Brodmann \cite{Brodmann}
who showed that for any ideal $J$ in a Noetherian ring, there exists an
integer $a$ such that
$$\bigcup_{s=1}^\infty \Ass(R/J^s) = \bigcup_{s=1}^a \Ass(R/J^s).$$
Yet little is known about the place at which stabilization occurs (i.e., the minimal such integer $a$). Hoa \cite{Hoa}
provided a rough bound on $a$ when $J$ is a monomial ideal in a polynomial ring. When $J = J(\HH)$
is the cover ideal of a hypergraph, we give a lower bound for $a$ in terms of $\chi(\HH)$ in 
Corollary \ref{stablebound}, thus giving a partial
answer to (3).

Sturmfels and Sullivant \cite{SS} gave one answer to question (4)
by classifying perfect graphs in terms of the generators
of the $s$-th secant ideal of $I(G)$.
We provide new algebraic characterizations, 
independent of the SPGT, thus giving new answers to (4):

\begin{theorem} [Theorem ~\ref{perfectgraphs}]
Let $G$ be a simple graph with cover ideal $J$. The following are equivalent:
\begin{enumerate}
\item $G$ is perfect.
\item For all $s \ge 1$, $P=(x_{i_1}, \dots, x_{i_r}) \in \Ass(R/J^s)$ if and only if
the induced graph on
$\{x_{i_1},\dots,x_{i_r}\}$ is a clique of size $1< r \le s+1$ in $G$.
\item $J^s$  has the saturated chain condition for associated primes for all $s \geq 1$.
\end{enumerate}
\end{theorem}

As a consequence of Theorem~\ref{perfectgraphs}, we prove that if $J$ is the cover 
ideal of a perfect graph, then $\Ass(R/J^s) \subseteq \Ass(R/J^{s+1})$ for all 
$s \ge 1$, giving us a condition for question (5). We also note in 
Corollary~\ref{assprimescliques} that primes corresponding to cliques, odd holes, 
and odd antiholes always persist, regardless of whether the underlying graph 
is perfect, giving a partial answer to (5).
Observe that we have also identified a new infinite family of ideals which
have the saturated chain condition for associated primes, answering (6).  
Very few other families of ideals with this  property are known (cf. \cite{HostenThomas}). 

\noindent {\bf Acknowledgments.} This project began during a Research in Teams
week at Banff International Research Station. The genesis of our main results was a
computer experiments using CoCoA \cite{C} and {\it Macaulay 2} \cite{M2}.  We 
thank Hailong Dao, Ch\'inh Ho\`ang, Craig Huneke, Jeremy Martin,
and Jeffrey Mermin for their helpful conversations, and  Bjarne Toft and  Anders Sune Pedersen 
for information about critically chromatic graphs.
In addition, we thank an anonymous
referee of a previous paper who suggested the idea of the $s$-th expansion of a graph.
The first author is partially
supported by an NSA Young Investigator's Grant and an Oklahoma State University Dean's
Incentive Grant. The second author is partially supported by the Board of Regents grant
LEQSF(2007-10)-RD-A-30 and Tulane's Research Enhancement Fund. The third author was supported
by NSERC.


\section{Preliminaries: some hypergraph theory and associated primes} \label{s.prelim}

We explain some terminology that we shall use in the rest of the paper.

\subsection{Cover ideals, secant ideals and Alexander duality}

Throughout, $\HH$ is a finite simple hypergraph on the vertex set $V_\HH =\{x_1,\dots,x_n\}$ and
with the edge set $\mathcal{E}_\HH = \{E_1,\dots,E_t\}$. We assume that $\HH$ has no isolated vertices
and that each $|E_i| \ge 2$. When the $E_i$s all have cardinality two, then $\HH$ is a
finite simple graph. In this case, we use $G$ in place of $\HH$. We shall use $I=I(\HH)$
for the edge ideal of $\HH$ and $J=J(\HH)$ for the cover ideal of $\HH$, which is generated
 by the squarefree monomials corresponding to the minimal vertex covers of $\HH$.
A \textbf{vertex cover} of $\HH$ is a subset $W$ of $V_{\HH}$ such that if
$E \in \mathcal{E}_{\HH}$, then $W \cap E \not = \emptyset$. 
A vertex cover is \textbf{minimal} if
no proper subset is also a vertex cover.

\begin{notation}\label{notation}
We shall use the following notation throughout this paper.  For any subset $W \subseteq V$
of the vertices, we use $\bm_W$ to denote the monomial $\prod_{x \in W} x$, and we use $x_W$ to
denote the ideal $(x ~|~ x \in W)$. Hence $\bm_V$ denotes $x_1\cdots x_n$, and $x_V$ denotes
the maximal homogeneous ideal in $R$. Moreover, if $\mathbf{a} = (a_1,\dots,a_n) \in \mathbb{N}^n$,
we write $\mathbf{x}^\mathbf{a}$ for $x_1^{a_1} \cdots x_n^{a_n}$.
\end{notation}

The ideals $I=I(\HH)$ and $J=J(\HH)$ are squarefree Alexander duals of each other. Later, we
shall need the generalized Alexander duality for arbitrary monomial ideals. The best reference for
this topic is Miller and Sturmfels's book \cite{MS}. Let $\mathbf{a}$ and $\mathbf{b}$ be
vectors in $\mathbb{N}^n$ such that $b_i \le a_i$ for each $i$. As in \cite[Definition 5.20]{MS},
we define the vector $\mathbf{a \setminus b}$ to be the vector whose $i$-th entry is given by
\[
a_i \setminus b_i = \left\{
\begin{array}{l l}
 a_i+1-b_i & \text{if } b_i \ge 1 \\
 0 & \text{if } b_i=0. \\
\end{array}
\right.
\]

\begin{definition} \label{generalalexanderdual}
Let $\mathbf{a} \in \mathbb{N}^n$, and let $I$ be a monomial ideal such that all elements of
$\mathcal{G}(I)$, the minimal generators of $I$, divide $\mathbf{x}^\mathbf{a}$. The
\textbf{Alexander dual} of $I$ \textbf{with respect to} $\mathbf{a}$ is the ideal
\[ I^{[\mathbf{a}]} = \bigcap_{\mathbf{x}^\mathbf{b} \in \mathcal{G}(I)} \, (x_1^{a_1 \setminus b_1}, \dots, x_n^{a_n \setminus b_n}).\]
\end{definition}

\begin{definition}
A monomial ideal in $R = k[x_1,\ldots,x_n]$ of the form $\mf^{{\bf b}} = (x_i^{b_i} ~|~ b_i \geq 1)$ with ${\bf b} = (b_1, \dots, b_n) \in \N^n$ is called an {\bf irreducible ideal}.  An {\bf irreducible decomposition} of a monomial ideal $I$ is an expression of the form
\[ I = \mf^{{\bf b}_1} \cap \cdots \cap \mf^{{\bf b}_r}~~\mbox{for some vectors ${\bf b}_1,\ldots,{\bf b}_r \in \N^n$.} \]
The decomposition
is {\bf irredundant} if none of the $\mf^{{\bf b}_i}$ can be omitted.
\end{definition}

The ideals $x_W$ introduced in Notation \ref{notation} are examples
of irreducible ideals;  in particular, $x_W = {\bf m}^{{\bf b}}$,
where $b_i = 1$ if $x_i \in W$, and 0 otherwise.
There is a bijection between the irredundant
irreducible decomposition of a monomial ideal $I$ and the generators of $I^{[{\bf a}]}$,
the Alexander dual of $I$ with respect to ${\bf a}$ (see \cite[Theorem 5.27]{MS}).

\begin{theorem}  \label{correspondence}
Let $I$ be a monomial ideal whose minimal generators all divide
${\bf x}^{\bf a}$.  Then $I$ has a unique irredundant irreducible decomposition, which is
given by
\[I = \bigcap_{{\bf x}^{{\bf a} \setminus {\bf b}} ~\mbox{is a minimal generator
of $I^{[{\bf a}]}$}} \mf^{{\bf b}}.\]
\end{theorem}

Our sources for joins and secant ideals are the papers
\cite{SU,SS}.

\begin{definition} Let $I_1, \dots, I_s$ be ideals in $R = k[\mathbf{x}]$, where
$\mathbf{x} = \{x_1, \dots, x_n\}$. Introduce $s$ new groups of variables
$\mathbf{y}_i = \{y_{i1}, \dots, y_{in}\}$ for $i=1,\ldots,s$, and let $I_i(\mathbf{y}_i)$ be the image of
$I_i$ in $k[\mathbf{x}, \mathbf{y}_1, \dots, \mathbf{y}_s]$ under the map that sends $x_j$
to $y_{ij}$. The \textbf{join} of $I_1, \dots, I_s$, denoted by $I_1 * \cdots * I_s$, is
defined to be the elimination ideal
$$\Big( I_1(\mathbf{y}_1) + \dots + I_s(\mathbf{y}_s) + (y_{1j} + \dots + y_{sj} - x_j ~|~ j = 1, \dots, n)
 \Big) \cap R.$$
The \textbf{$s$-th secant ideal} of an ideal $I \subseteq R$ is the $s$-fold join of $I$
with itself, that is,
$$I^{\{s\}} = I * \cdots * I.$$
\end{definition}

For a monomial ideal $I \subseteq k[\mathbf{x}]$, the \textbf{standard monomials} of $I$
are the monomials in $k[\mathbf{x}] \backslash I$. The following result of \cite{SS} will be used later.

\begin{theorem}[Proposition 2.4 of \cite{SS}] \label{standardmon}
Let $I_1, \dots, I_s$ be squarefree monomial ideals in $k[\mathbf{x}]$. 
Then $I_1 * \cdots * I_s$ is
also a monomial ideal in $k[\mathbf{x}]$, whose standard monomials
are precisely the products $m_1 \cdots m_s$, where
$m_i$ is a standard monomial of $I_i$.  This result is
independent of the characteristic of $k$.
\end{theorem}

\subsection{Terminology from (hyper)graph theory}

We shall frequently have occasion to investigate particular types of subhypergraphs.
If $P \subseteq V$, then the {\bf induced subhypergraph} of $\HH$ on $P$, denoted by $\HH_P$,
 is the hypergraph with vertex set $P$ and edge set $\mathcal{E}_{\HH_P} = \{ E \in \mathcal{E}_\HH ~|~ E \subseteq P\}$. 

When we specialize to graphs, we use some additional terminology.
A \textbf{cycle} in a simple graph $G$ is an alternating sequence of distinct vertices
and edges that we write as $C = x_{i_1}e_1x_{i_2}e_2 \cdots x_{i_{n-1}}e_{n-1}x_{i_n}e_nx_{i_1}$ in which the edge
$e_j$ connects the vertices $x_{i_j}$ and $x_{i_{j+1}}$ ($x_{i_{n+1}} = x_{i_1}$) for all $j$.
In this case,  we call $C$ an {\bf $n$-cycle}. We 
write a cycle simply as $x_{i_1}x_{i_2} \cdots x_{i_n}x_{i_1}$, omitting the edges.
A {\bf chord} is an edge that joins two nonadjacent vertices in the cycle. We 
use $C_n$ to denote an $n$-cycle without any chords, and call $C_n$ an
{\bf induced cycle} since the induced graph on $\{x_{i_1},x_{i_2}, \dots, x_{i_n}\}$ contains
only the edges and vertices in the cycle. If an induced cycle has an odd (resp., even)
number of vertices, it is called an {\bf odd} (resp., {\bf even}) {\bf cycle}.  An
odd induced cycle of length at least five is called an {\bf odd hole}.

\begin{definition} \label{chi}
Let $\HH = (V, \mathcal{E})$ be a hypergraph. A {\bf $d$-coloring} of $\HH$ is any
partition of $V = C_1 \cup \cdots \cup C_d$ into $d$ disjoint sets such that for every
$E \in \mathcal{E}$, we have $E \not\subseteq C_i$ for all $i = 1,\dots,d$. (In the case
of a graph $G$, this simply means that any two vertices connected by an edge receive
different colors.) The $C_i$s are called the {\bf color classes}. Each color class $C_i$
is an {\bf independent set}, meaning that $C_i$ does not contain any edge of the hypergraph. 
The {\bf chromatic number} of $\HH$, denoted $\chi(\HH)$,
is the minimal $d$ such that $\HH$ has a $d$-coloring.
\end{definition}

The following definition (for the case of graphs) is found in Diestel \cite[page 134]{D}.

\begin{definition}
A hypergraph $\HH$ is called {\bf critically $d$-chromatic} if $\chi(\HH) = d$,
but for every vertex $x \in V$, $\chi(\HH \setminus \{x\}) < d$, where $\HH \setminus \{x\}$
denotes the hypergraph $\HH$ with $x$ and all edges containing $x$ removed.
\end{definition}

\begin{example} Let $G = C_n$ be any odd cycle (that is, $n \geq 3$ is odd).
Then $G$ is a critically 3-chromatic graph since $\chi(G) = 3$, but
if we remove any vertex $x \in V$, $\chi(G\setminus \{x\}) = 2$.
\end{example}

\subsection{Associated Primes}

Associated primes are a classical notion:

\begin{definition} \label{d.assocprime} Let $M$ be an $R$-module.
 A prime ideal $P$ is called an {\bf associated prime} of $M$ if $P=\Ann(m)$,
the annihilator of $m$, for some $m \in M$. The set of all associated primes of $M$
is denoted $\Ass(M)$.
\end{definition}

Throughout this paper, we will only be considering the associated primes
of monomial ideals.
Hence, any associated prime of $R/J^d$ will
have the form $P = (x_{i_1},\dots,x_{i_r})$ for some subset $\{x_{i_1},\dots,x_{i_r}\} \subseteq V_\HH$.
 We shall slightly abuse notation in the following way:  we shall use $P$ to denote both
a subset $\{x_{i_1},\dots,x_{i_r}\}$ of $V$ and the prime monomial ideal $(x_{i_1},\dots,x_{i_r})$.
We hope it is clear from the context which form we are using.

We end this section with a technical lemma that allows us to simplify our future arguments.
To determine if a prime ideal is associated to $J(\HH)^d$,
it is enough to determine if the maximal ideal is associated to a power of the cover ideal of an
induced subhypergraph.

\begin{lemma}\label{localize}
Let $\HH$ be a finite simple hypergraph on $V = \{x_1,\dots,x_n\}$ with cover ideal
$J(\HH) \subseteq R = k[x_1,\dots,x_n]$.  Then
\[P = (x_{i_1},\dots,x_{i_r}) \in \Ass(R/J(\HH)^d) \Leftrightarrow P = (x_{i_1},\dots,x_{i_r}) \in \Ass(k[P]/J(\HH_P)^d),\] where $k[P] = k[x_{i_1},\dots,x_{i_r}]$, and $\HH_P$ is the induced hypergraph
of $\HH$ on the vertex set $P = \{x_{i_1},\dots,x_{i_r}\} \subseteq V$.
\end{lemma}

\begin{proof}  To simplify our notation, we take $P = (x_1,\dots,x_m)$, and hence
$k[P] = k[x_1,\dots,x_m]$.

$(\Rightarrow)$  Suppose that $P = (x_1,\dots,x_m) \in \Ass(R/J(\HH)^d)$. Then there
exists a monomial $T$ such that $J(\HH)^d:T = P$. We can rewrite $T$ as $T = T_1T_2$,
where $T_1$ is a monomial in $k[P]$ and $T_2$ is a monomial in $\{x_{m+1},\dots,x_n\}$.

For any monomial $u$ in the variables $\{x_{m+1},\dots,x_n\}$, we claim that
$J(\HH)^d:Tu = J(\HH)^d:T$.  To see this, first note that $Tu \not \in J(\HH)^d$,
for if it were, then $u \in J(\HH)^d:T = P$, which is false since $u \not\in P$.
For any $x_j \in P$, $(Tu)x_j = (Tx_j)u \in J(\HH)^d$ since $Tx_j \in J(\HH)^d$.  So $P \subseteq J(\HH)^d:Tu$.  Finally, take any monomial $w \in R$ such that $w \in J(\HH)^d:Tu$.
If $w$ is a monomial only in the variables $\{x_{m+1},\dots,x_n\}$, then
$(Tu)w = T(uw) \in J(\HH)^d$ implies that $uw \in P$, which is again a contradiction
since neither $u$ nor $w$ is divisible by any of $\{x_1,\dots,x_m\}$.  So $J(\HH)^d:Tu = P$.

As a consequence of the above discussion, we can multiply $T$ by a suitable monomial $u$ in the
variables $\{x_{m+1},\dots,x_n\}$ so that $T = T_1T_2$ with $T_2 = (x_{m+1}\cdots x_n)^dT'_2$. That is,
\[J(\HH)^d:T_1(x_{m+1}\cdots x_n)^dT'_2 = (x_1,\dots,x_m).\]

We now show that $J(\HH_P)^d:T_1 = (x_1,\dots,x_m)$ in $k[P]$. First, we show that
$T_1 \not\in J(\HH_P)^d$. If it were, then there exists monomials $m_1,\dots,m_d \in J(\HH_P)$ such
that  $T_1 = m_1\cdots m_dM$ for some monomial $M \in k[P]$. But then
\begin{eqnarray*}
T &=& T_1(x_{m+1}\cdots x_n)^dT'_2  = (m_1\cdots m_dM) (x_{m+1}\cdots x_n)^dT'_2 \\
& =& [m_1(x_{m+1}\cdots x_n)][m_2(x_{m+1}\cdots x_n)]\cdots [m_d(x_{m+1}\cdots x_n)]MT'_2.
\end{eqnarray*}
Note that for each $i =1,\dots,d$, $m_i(x_{m+1}\cdots x_n)$ is a vertex cover of $\HH$
since $m_i$ covers $\HH_P$, and $x_{m+1}\cdots x_n$ covers the remaining edges. The
above expression thus implies that $T \in J(\HH)^d$, a contradiction. So $T_1 \not\in J(\HH_P)^d$.

For any $x_i \in P$, we have $Tx_i \in J(\HH)^d$.  Thus, there exists $n_1,\ldots,n_d \in J(\HH)$
such that
\begin{eqnarray*}
Tx_i & = & T_1T_2x_i = n_1 \cdots n_dN =  n_{1,1}n_{1,2} \cdots n_{d,1}n_{d,2}N,
\end{eqnarray*}
where we have written each $n_i$ as $n_i = n_{i,1}n_{i,2}$ with $n_{i,1}$ a monomial in $k[P]$
and $n_{i,2}$ a monomial in the remaining variables.  Thus, if we compare the monomials in
$k[P]$ on both sides of the above expression, we get $(n_{1,1}\cdots n_{d,1}) \mid T_1x_i.$
But each $n_{i,1}$ corresponds to a vertex cover of $\HH_P$ and thus, is an element of $J(\HH_P)$.
So $T_1x_i \in J(\HH_P)^d$. Thus the maximal ideal $(x_1,\dots,x_m) \subseteq J(\HH_P)^d:T_1$,
and since $T_1 \not\in J(\HH_P)^d$, we have $J(\HH_P)^d:T_1 = (x_1,\dots,x_m)$, as desired.

($\Leftarrow$) Suppose $P = (x_1,\dots,x_m) \in \Ass(k[P]/J(\HH_P)^d)$. Thus there exists a
monomial $T \in k[P]$ with $T \not\in J(\HH_P)^d$ such that $J(\HH_P)^d:T = P$. We will
show that \[J(\HH)^d:T(x_{m+1}\cdots x_n)^d = (x_1,\dots,x_m).\]

We first observe that $T(x_{m+1}\cdots x_n)^d \not\in J(\HH)^d$.  If it were, then there
exist $n_1,\dots,n_d \in J(\HH)$ such that $T(x_{m+1}\cdots x_n)^d = n_1\cdots n_dM.$
Rewriting each $n_i$ as $n_i = n_{i,1}n_{i,2}$, where $n_{i,1}$ is a monomial in $k[P]$, and $n_{i,2}$
is a monomial in the variables $\{x_{m+1},\dots,x_n\}$, we have $n_{1,1}\cdots n_{d,1} \mid T$.
 But each $n_{i,1}$ corresponds to a
vertex cover of $\HH_P$ since $n_i$ corresponds to a vertex cover of $\HH$.
This means that $T \in J(\HH_P)^d$, a contradiction.

Now let $x_i$ be a generator of $P$.  In the ring $k[P]$, $Tx_i = m_1\cdots m_dM$ with
$m_i \in J(\HH_P)$ for each $i$. But then in $R$,
\[T(x_{m+1}\cdots x_n)^dx_i = [m_1(x_{m+1}\cdots x_n)]\cdots [m_d(x_{m+1}\cdots x_n)]M.\]
For each $i = 1,\dots,d$, the monomial $m_i(x_{m+1}\cdots x_n)$ corresponds to a vertex
cover of $\HH$ and thus is an element of $J(\HH)^d$. Hence $P \subseteq J(\HH)^d:T(x_{m+1}\cdots x_n)^d$.
 For the reverse inclusion, consider any monomial $w \in J(\HH)^d:T(x_{m+1}\cdots x_n)^d$.
 If there exists some variable $x_i \in \{x_{m+1},\dots,x_n\}$ such that $x_i \mid w$,
then $\frac{w}{x_i} \in J(\HH)^d:T(x_{m+1}\cdots x_n)^d$.  Indeed,
if $w = w'x_i$, then  $T(x_{m+1}\cdots x_n)^d(w'x_i) = m_1\cdots m_dM,$
and because each $m_i$
 is squarefree and can be divisible by at most one $x_i$, we have $x_i \mid M$.  So \[T(x_{m+1}\cdots x_n)^dw' = m_1\cdots m_d\left(\frac{M}{x_i}\right),\] whence $w' \in  J(\HH)^d:T(x_{m+1}\cdots x_n)^d$.
We can now reduce to the
case that $w$ is a monomial
only in the variables of $\{x_1,\dots,x_m\}$.  But this just means that $w \in P$, as
 desired.
\end{proof}


\section{An algebraic method to compute the chromatic number} \label{s.chromatic}

We give an algebraic description of  $\chi(\HH)$.
Sturmfels and Sullivant also gave an algebraic description of $\chi(\HH)$ 
(see \cite[Theorem 3.1]{SS}), which is equivalent to our result
via Alexander duality.  Our proof has the advantage that it 
avoids the language of secant ideals used in \cite{SS}, and it
can be easily adapted to compute the $b$-fold chromatic number.

The following fact, which follows directly from the definitions, is stated as a lemma so that we
may refer back to it throughout the paper.
The proof is omitted.

\begin{lemma}\label{color2vertexcover}
Let $\HH = (V, \mathcal{E})$ be a simple hypergraph, and let $C \subseteq V$ be a
subset of the vertices. Then $C$ is an independent set if and only if $V \backslash C$
is a vertex cover of $\HH$.
\end{lemma}

\begin{theorem}\label{colortheorem}
Let $\HH$ be a finite simple hypergraph on $V = \{x_1,\dots,x_n\}$ with cover ideal $J$.
Then $\bm_V^{d-1} \in J^d$ if and only if $\chi(\HH) \le d$.  In particular,
\[\chi(\HH) = \min\{d ~|~ \bm_V^{d-1} \in J^d \}.\]
\end{theorem}

\begin{proof}
($\Rightarrow$)  Suppose that $\bm_V^{d-1} \in J^d$.  Then there exists $d$ minimal vertex
covers $W_1,\dots,W_d$ (not necessarily distinct) such that
$ \bm_{W_1}\cdots \bm_{W_d} \mid (x_1\cdots x_n)^{d-1} = \bm_V^{d-1}.$
 For each $x_i \in \{x_1,\dots,x_n\}$, there exists some $W_j$ such that $x_i \not\in W_j$;
otherwise, if $x_i \in W_j$ for all $1 \leq j \leq d$, then the power of $x_i$ is $d$
in $\bm_{W_1}\cdots \bm_{W_d}$, from which it follows that $\bm_{W_1}\cdots \bm_{W_d}$
cannot divide $\bm_V^{d-1}$, a contradiction.

Now form the following $d$ sets:
\begin{eqnarray*}
C_1 & = & V \setminus W_1 \\
C_2 & = & (V \setminus W_2) \setminus C_1 \\
C_3 & = & (V \setminus W_3) \setminus (C_1 \cup C_2) \\
& \vdots & \\
C_d & = & (V \setminus W_d) \setminus (C_1 \cup \cdots \cup C_{d-1}).
\end{eqnarray*}
It suffices to show that $C_1,\dots,C_d$ form a $d$-coloring of $\HH$.
We first note that by construction, the $C_i$s are pairwise disjoint.
As well, because each $C_i \subseteq V \setminus W_i$, each $C_i$ is an independent set.
 So it remains to show that $V = C_1 \cup \cdots \cup C_d$.  If $x \in V$,
there exists some $W_j$ such $x \not\in W_j$, whence $x \in V \setminus W_j$.
Hence $x \in C_j$ or $x \in (C_1 \cup \cdots \cup C_{j-1})$. 

($\Leftarrow$)  Let $C_1 \cup \cdots \cup C_{\chi(\HH)}$ be a $\chi(\HH)$-coloring of $\HH$.
 By Lemma \ref{color2vertexcover}, we know that
\[Y_i = C_1 \cup \cdots \cup \widehat{C_i} \cup \cdots \cup C_{\chi(\HH)}
~~\mbox{for each $i=1,\dots,\chi(\HH)$}\] is a vertex cover of $\HH$, and hence
$\bm_{Y_i} \in J$ for $i=1,\dots,\chi(\HH)$.
It follows that
 \[\prod_{i=1}^{\chi(\HH)} \bm_{Y_i} = \Big(\prod_{i=1}^{\chi(\HH)} \bm_{C_i}\Big)^{\chi(\HH)-1} = (x_1\cdots x_n)^{\chi(\HH)-1} = \bm_V^{\chi(H)-1} \in J^{\chi(\HH)}.\]
Thus $\bm_V^{d-1} = \bm_{V}^{\chi(\HH)-1}\bm_V^{d-\chi(\HH)} \in J^{\chi(\HH)}J^{d-\chi(\HH)} = J^{d}$ because $\bm_V \in J$.
\end{proof}

\begin{remark}
The \emph{Macaulay 2} 
{\tt EdgeIdeals} package \cite{EdgeIdeals} uses Theorem \ref{colortheorem} to compute $\chi(\HH)$
algebraically.  To compute this number, one needs $J = I(\HH)^{\vee}$,
which is equivalent to finding all the minimal vertex covers of $\HH$. Because $J$ is a
monomial ideal, the operations of computing $J^d$ and ideal membership
are both fairly simple; the bottleneck of this procedure
is computing $J$, causing the algorithm to be exponential time in general.
\end{remark}

\begin{example}
A graph $G$ is a {\bf bipartite graph} if there is a bipartition $V = V_1 \cup V_2$ such
that every edge of $G$ has one vertex in $V_1$ and the other vertex in $V_2$.
A graph $G$ with at least one edge is bipartite if and only if $\chi(G) = 2$
if and only if $\bm_V \in J^2$.
\end{example}

A modification to Theorem \ref{colortheorem} allows us
to calculate the $b$-fold chromatic number.

\begin{definition} A {\bf $b$-fold coloring} of a graph $G$ is an assignment
to each vertex a set of $b$ distinct colors such that adjacent vertices
receive disjoint sets of colors.  The {\bf $b$-fold
chromatic number} of $G$, denoted $\chi_b(G)$, is 
the minimal number of colors needed in a $b$-fold coloring of $G$.
\end{definition}

When $b = 1$, then $\chi_b(G) = \chi(G)$.  We can now generalize Theorem \ref{colortheorem}.

\begin{theorem}\label{bfold}
Let $G$ be a finite simple graph with
cover ideal $J$. Then \[\chi_b(G) = \min\{d ~|~ \bm_V^{d-b} \in J^d \}  ~~\mbox{where
$V = \{x_1,\ldots,x_n\}$}.\]
\end{theorem} 

\begin{proof} Just as in Theorem \ref{colortheorem}, we will show
that $\chi_b(G) \leq d$ if and only if $\bm_V^{d-b} \in J^d$.

First, suppose that $\bm_V^{d-b} \in J^d$.  Thus, there exist $d$ vertex covers
$W_1,\ldots,W_d$ such that $\bm_V^{d-b} = \bm_{W_1}\cdots \bm_{W_d}M$.  For each $i=1,\ldots,d$,
let $C_i = V \setminus W_i$.  For each $j = 1,\ldots,n$, $x_j$ appears in at most
$d-b$ of the $W_i$s, or equivalently, $x_j$ appears in at least $b$
of $C_i$s.  Say that $x_j$ appears in $C_{i_1},\ldots,C_{i_b},C_{i_{b+1}},\ldots,C_{i_a}$.  We associate
to $x_j$ the colors $\{i_1,\ldots,i_b\}$.  We claim that this coloring is a $b$-fold
coloring of $G$.  Indeed, each vertex has received $b$ colors, so it suffices
to show that adjacent vertices receive disjoint sets of colors.  So,
suppose $x_jx_k \in E_G$, and $x_i$ is colored $\{i_1,\ldots,i_b\}$ and
$x_k$ is colored $\{l_1,\ldots,l_b\}$.  If there is a $p \in \{i_1,\ldots,i_b\}
\cap \{l_1,\ldots,l_b\}$, this means that $x_j \in C_p$ and $x_k \in C_p$.
But $C_p = V \setminus W_p$ is an independent set,
and so $x_j$ and $x_k$ cannot both be in $C_p$.  Thus $\chi_b(G) \leq d$.

For the converse direction, suppose that $\chi_b(G) = a \leq d$, and assume
that $G$ has already been given a $b$-fold coloring using $a$ colors,
say $\{1,\ldots,a\}$.  For $i=1,\ldots,a$, let 
\[W_i = \{x_j \in V ~|~ \mbox{$x_j$ does not receive color $i$}\}.\]
Since the set of vertices in a $b$-fold coloring that receive the color $i$ form
an independent set, the set $W_i$ forms a vertex cover, and hence
$\bm_{W_i} \in J$ for $i=1,\ldots,a$.  Thus $\bm_{W_1}\cdots \bm_{W_a} \in J^a$.
Since each vertex $x_j$ receives exactly $b$ distinct colors, 
each $x_j$ is not in $b$ of the $W_i$s, or equivalently, $x_j$
is in exactly $a-b$ of the  $W_i$s.  But this implies that
\[\bm_{W_1}\cdots \bm_{W_a} = \bm_V^{a-b} \in J^a.\]
Since $\bm_V \in J$, we thus have $\bm_V^{a-b}\bm_V^{d-a} = \bm_V^{d-b} \in J^aJ^{d-a} \subseteq J^d$,
as desired.
\end{proof}

\begin{example}  Let $G = C_5$ be the $5$-cycle.  A $2$-fold coloring of $G$
is given below:
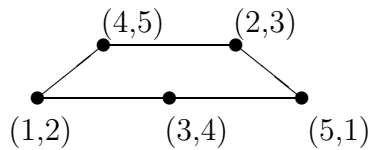
\[
\begin{picture}(110,60)(0,0)
\put(0,20){\circle*{5}}
\put(-11,4){(1,2)}
\put(50,20){\circle*{5}}
\put(48,4){(3,4)}
\put(100,20){\circle*{5}}
\put(102,4){(5,1)}
\put(75,40){\circle*{5}}
\put(74,45){(2,3)}
\put(25,40){\circle*{5}}
\put(24,45){(4,5)}
\put(0,20){\line(1,0){100}}
\put(25,40){\line(1,0){50}}
\put(0,20){\line(5,4){25}}
\put(75,40){\line(5,-4){25}}
\end{picture}\]
We see that we need at least $5$ colors to find a $2$-fold coloring of $G$.  In
fact, we have $\chi_2(G) = 5$.  Thus, by Theorem \ref{bfold},
we will have $(x_1x_2x_3x_4x_5)^{5-2} \in J(G)^5$, which a computer
algebra system can easily verify.   
\end{example}

\begin{remark}  The {\bf fractional chromatic number} of a graph $G$, denoted $\chi_f(G)$,
is defined in terms of the numbers $\chi_b(G)$.  Precisely,
\[\chi_f(G) := \lim_{b \rightarrow \infty} \frac{\chi_b(G)}{b}.\]
It can be shown that there exists an integer $b$ such
that $\chi_f(G) = \chi_b(G)/b$, so Theorem \ref{bfold} may give some insight
into the value of $\chi_f(G)$.  This is particularly interesting since
$\chi_f(G)$ can also be viewed as a solution to a linear programming
problem.  See the book of Scheinerman and Ullman \cite{SU1} for more on the fractional
chromatic number.
\end{remark}

\section{Irreducible decompositions and associated primes of the powers of cover ideals} \label{s.assoc}

In this section, we explore the graph theoretic information
encoded into the associated primes of $R/J^s$ as $s$ varies.  We will
show that the irredundant irreducible decomposition of $J^s$ describes
the critically $(s+1)$-chromatic induced subhypergraphs of the $s$-th
expansion of the hypergraph associated to $\HH$.  
Our strategy to obtain an irreducible decomposition of $J^s$ 
is to first describe the minimal 
generators of the generalized Alexander dual $(J^{s})^{[{\bf s}]}$,
and then use Theorem \ref{correspondence}.  

To find the minimal generators of $(J^{s})^{[{\bf s}]}$, we use the following result of Sturmfels and Sullivant \cite{SS}, which
was first proved for graphs, but also holds for hypergraphs.

\begin{theorem} \label{ss-maintheorem}
Let $\HH$ be a finite simple hypergraph with edge ideal $I(\HH)$. Then $I(\HH)^{\{s\}}$ is
generated by the squarefree monomials $\bm_W$ such that $\HH_W$, the induced
subhypergraph on the vertices $W$, is not $s$-colorable. That is,
\[ I(\HH)^{\{s\}} = (\bm_W ~\big|~ \chi(\HH_W) > s),\] and its minimal generators are
monomials $\bm_W$ such that $\HH_W$ is critically $(s+1)$-chromatic.
\end{theorem}

\begin{proof}
When $\HH$ is a graph, this is precisely \cite[Theorem 3.2]{SS}. We need only check
that the proof goes through for hypergraphs. \cite[Theorem 3.2]{SS} relies on the
proof of \cite[Proposition 3.1]{SS}, and it is easy to see that the argument still
 works after substituting the generalized definitions of coloring and independent
set for hypergraphs. The last step in the proof requires that $I(\HH)^{\{s\}}$ be radical,
but $I^{\{s\}}$ is radical when $I$ is any squarefree monomial ideal, so using edge ideals of
hypergraphs instead of graphs is no problem.
\end{proof}

It follows from \cite[Corollary 2.7]{SS} that the generators
of $I(\HH)^{\{s\}}$ are precisely the squarefree monomial generators of $(J^s)^{[{\bf s}]}$.
We wish to interpret the remaining generators of $(J^s)^{[{\bf s}]}$ in terms
of the hypergraph $\HH$.  To carry out this program, we must
make use of the $s$-th expansion of a (hyper)graph.
We thank an anonymous referee of a previous paper for introducing us to this idea
and for suggesting the statement of Theorem~\ref{assoc-depolar} in the case of graphs.

\begin{definition} \label{expansion}
Let $\HH = (V,\mathcal{E})$ be a hypergraph over the vertices $V= \{x_1, \dots, x_n\}$.
For each $s$, we define the $s$-th \textbf{expansion} of $\HH$ to be the hypergraph
obtained by replacing each vertex $x_i \in V$ by a collection $\{x_{ij} ~|~ j = 1, \dots, s\}$,
and replacing $\mathcal{E}$ by the edge set that consists of edges $\{x_{i_1l_1}, \dots, x_{i_rl_r}\}$
whenever $\{x_{i_1}, \dots, x_{i_r}\} \in \mathcal{E}$ and edges
$\{x_{il}, x_{ik}\}$ for $l \not= k$. We denote this hypergraph by $\HH^s$. The new
variables $x_{ij}$ are called the {\bf shadows} of $x_i$. The process of setting $x_{il}$
to equal to $x_i$ for all $i$ and $l$ is called the \textbf{depolarization}.
\end{definition}

\begin{remark}
Although the expansion construction is a common construction in graph theory,
there does not appear to be any consistent
terminology.   In some cases, the expansion of a vertex is called the replication
of a vertex.  There is a similar construction, called either the duplication or
parallelization
of a vertex, where the definition is the same as above except we
do not add the edges $\{x_{il}, x_{ik}\}$ for $l \not= k$, i.e., the
vertices  $\{x_{ij} ~|~ j = 1, \dots, s\}$ form an independent set instead
of a clique as in our case.  Mart\'{i}nez-Bernal, Renter\'{i}a,  and Villarreal
\cite{MRV} showed how the parallelization of a hypergraph
$\HH$ is related to the Rees algebra of $I(\HH)$.
\end{remark}

Let $\mathbf{s} = (s, \dots, s) \in \N^n$. Recall that for a subset $W \subseteq V$, $x_W$
denotes the ideal $(x ~|~ x \in W)$. 
The following result is a higher-dimensional
analog of \cite[Theorem 3.2]{SS};  not only does it apply to
hypergraphs, but it gives a graph theoretic interpretation
for all the generators of $(J^{s})^{[{\bf s}]}$, not
just the squarefree generators.

\begin{theorem} \label{assoc-depolar}
Let $\HH = (V_\HH, \mathcal{E}_\HH)$ be a finite simple hypergraph 
with cover ideal $J = J(\HH)$.
For any $s \ge 1$, we have
\[(J^s)^{[\mathbf{s}]} = (\overline{\bm_T} ~\big|~ \chi(\HH^s_T) > s )\]
where $\overline{\bm_T}$, the depolarization of $\bm_T$, is obtained by setting
$x_{ij}$ to be $x_i$ for any $i,j$.
\end{theorem}

\begin{proof} It follows from Theorem~\ref{ss-maintheorem} that
\[I(\HH^s)^{\{s\}} = (\bm_T ~\big|~ \chi(\HH^s_T) > s ).\]
Thus, it suffices to show that the depolarization of $I(\HH^s)^{\{s\}}$ equals
$(J^s)^{[{\bf s}]}$. Furthermore, it is clear from the construction of $(J^s)^{[{\bf s}]}$ that for any $i$,
the highest power of $x_i$ in any minimal generator of $(J^s)^{[{\bf s}]}$ is at most $s$. This
implies that if $x_i$ appears in a monomial $M$ with power bigger than $s$ then
$M \not\in (J^s)^{[{\bf s}]}$ if and only if $M/x_i \not\in (J^s)^{[{\bf s}]}$. Hence, it is enough to
show that a monomial $M$, in which each $x_i$ appears with power at most $s$, is not in
$(J^s)^{[\mathbf{s}]}$ if and only if there exists a standard monomial of $I(\HH^s)^{\{s\}}$ that
depolarizes to $M$.

Let $\mathcal{V}$ be the collection of all minimal vertex covers of $\HH$. Then
\[I(\HH) = \bigcap_{U \in \mathcal{V}} x_U \text{ and } J = J(\HH) = (\bm_U ~\big|~ U \in \mathcal{V} ).\]
Therefore,
$J^s = \left( \prod_{i=1}^s \bm_{U_i} ~\big|~ U_1,\dots,U_s \in \mathcal{V} \right),$ and
by Theorem \ref{generalalexanderdual} we
 have
\[ (J^s)^{[\mathbf{s}]} = \bigcap_{U_1, \dots, U_s \in \mathcal{V}} (x_i^{s+1-k_i} ~|~ x_i \text{ belongs to
exactly } k_i > 0 \text{ vertex covers among } U_1, \dots, U_s).\]

Let $M \not\in (J^s)^{[{\bf s}]}$ be a monomial in which each $x_i$ appears with power at most $s$.
It follows from the description of $(J^s)^{[{\bf s}]}$ above that there exists a collection of
vertex covers $\{U_1, \dots, U_s\} \subseteq \mathcal{V}$ such that for any $x_i \in V_\HH$,
if $x_i$ belongs to exactly $k_i > 0$ vertex covers among $\{U_1, \dots, U_s\}$ then
$x_i^{s+1-k_i} \nmid M$, i.e., the highest power of $x_i$ in $M$ is at most $s-k_i$.

Let $Y = \{ x_i ~|~ x_i \text{ divides } M\}$. Let $\tilde{M}$ be the polarization of $M$
in the polynomial ring associated to $(\HH_Y)^s$, and let
$\tilde{Y} = \{x_{ij} ~|~ x_{ij} \text{ divides } \tilde{M}\}$. Then $\bm_{\tilde{Y}} = \tilde{M}$
depolarizes to $M$. For each $t = 1, \dots, s$, let $Y_t = Y \backslash U_t$. Then $Y_t$ is an
independent set of $\HH$. Moreover, if a vertex $x_i \in V_\HH$ is not in $\bigcup_{t=1}^s Y_t$,
then either $x_i \in V_{\HH} \setminus Y$ or
$x_i$ belongs to all of the vertex covers $\{U_1, \dots, U_s\}$. In the second case, it
follows that $x_i \nmid M$ since the highest
power of $x_i$ in $M$ is at most $s-s=0$.
Thus, in both cases we have $x_i \not\in Y$. Thus, $\bigcup_{t=1}^s Y_t = Y$.

By assigning color $t$ to the vertices in $Y_t$, we color $\HH_Y$ with $s$ colors,
where the coloring has the property that a
vertex may get many colors, and if $\{x_{i_1}, \dots, x_{i_r}\}$ is an edge in $\HH_Y$, then
there does not exist any color that is assigned to all the vertices $\{x_{i_1}, \dots, x_{i_r}\}$.
Observe that a vertex $x_i \in Y$ is assigned $s-k$ colors precisely when $x_i$ is contained
in exactly $k$ sets of the collection $\{U_1, \dots, U_s\}$; this is the case when at most
$s-k$ shadows of such an $x_i$ appear in $\tilde{M}$. We shall use (a subset of) these $(s-k)$ colors to assign a color to the shadows of $x_i$ appearing in $\tilde{M}$. It is easy to see that this is an $s$-coloring for $\HH^s_{\tilde{Y}}$. Thus, $\bm_{\tilde{Y}}$ is a standard monomial of
$I(\HH^s)^{\{s\}}$.

Conversely, suppose $M$ is a monomial, in which each $x_i$ appears with power at most $s$,
such that there exists a standard monomial $N$ of $I(\HH^s)^{\{s\}}$ that depolarizes to $M$.
It follows from Theorem~\ref{standardmon} that $N$ can be written as $\bm_{T_1} \dots \bm_{T_s}$,
 where $\bm_{T_1}, \dots, \bm_{T_s}$ are standard monomials of $I(\HH^s)$. Observe that if $\bm_U$
is a standard monomial of $I(\HH^s)$, and $x_{ij} \big| \bm_U$, then by replacing $x_{ij}$ by any
other shadow $x_{il}$ of $x_i$ in $U$ not already in $U$, we still get a standard monomial of $I(\HH^s)$.
Furthermore, each $x_i$ appears in $M$ with power at most $s$, so for each $i$, $N$
contains at most $s$ shadows of $x_i$ (counted with multiplicity). Thus, by replacing a
repeated shadow $x_{ij}$ of $x_i$ by collection of distinct shadows if necessary, we may
assume that $T_1, \dots, T_s$ are pairwise disjoint.

Let $T$ be the collection of variables
that appear in $N$, and let $Y$ be the collection of variables that appear in $M$. Observe that
since each $T_i$ is an independent set in $\HH^s$, we can assign color $j$ to the vertices in
$T_j$ to get an $s$-coloring for $\HH^s_T$. For each $x_i \in Y$, let $C_i$ be the set of colors
that shadows of $x_i$ appearing in $T$ obtained. We call $C_i$ the color class of $x_i$. For
each $t = 1, \dots, s$, form the set $U_t = Y \backslash \{x_i \in Y ~\big|~ t \in C_i\}$.
Observe that if $\{x_{i_1}, \dots, x_{i_r}\}$ is an edge in $\HH_Y$ then $\{x_{i_1l_1}, \dots, x_{i_rl_r}\}$
is an edge of $\HH^s_T$ for any $1 \le l_1, \dots, l_r \le s$, and so $\bigcap_{j=1}^r C_{i_j} = \emptyset$.
This implies that for any
$1 \le t \le s$, $\{x_{i_1}, \dots, x_{i_r}\} \not\subseteq \{x_i \in Y ~|~ t \in C_i\}$.
It follows that $U_t$ is a vertex cover of $\HH_Y$ for any $t = 1, \dots, s$.
Observe further that for each $x_i \in Y$, if $x_i$ belongs to exactly $k_i$ of the
vertex covers $\{U_1, \dots, U_s\}$, then the color class $C_i$ of $x_i$ has exactly
$s-k_i$ colors. Since the set of any two distinct shadows of $x_i$ is an edge in $\HH^s$,
these shadows cannot receive the same color. Thus, there are exactly $s-k_i$ shadows of $x_i$
appearing in $T$. This implies that, in this case, the highest power of $x_i$ appearing in $M$
is exactly $s-k_i$ for any $i$. Hence, $M$ is not in $(J^s)^{[\mathbf{s}]}$.
\end{proof}

By Theorem \ref{correspondence}, the irredundant irreducible decomposition of the ideal $J^s$ corresponds to the minimal generators of $(J^s)^{[\mathbf{s}]}$,
which  are precisely the depolarizations of monomials corresponding to critically $(s+1)$-chromatic induced subhypergraphs of $\HH^s$.  Thus we can
describe the elements of $\Ass(R/J^s)$ for all $s \geq 1$.

\begin{corollary}\label{classificationofass}
Let $\HH$ be a finite simple hypergraph with cover ideal $J$.  Then
\[(x_{i_1},\ldots,x_{i_r}) \in \Ass(R/J^s)\]
if and only if there exists some set $T$ with
\[\{x_{i_1,1},\ldots,x_{i_r,1}\} \subseteq T \subseteq \{x_{i_1,1},\ldots,x_{i_1,s},\ldots,x_{i_r,1},\ldots,x_{i_r,s}\}\]
such that $\HH_T^s$ is critically $(s+1)$-chromatic.
\end{corollary}

\begin{proof}
If $(x_{i_1},\ldots,x_{i_r}) \in \Ass(R/J^s)$, then
the irreducible ideal $(x_{i_1}^{a_{i_1}},\ldots,x_{i_r}^{a_{i_r}})$
with $a_{i_j} > 0$ appears in the irredundant irreducible decomposition of $J^s$.
By Theorem \ref{correspondence}, this means
that $x_{i_1}^{s+1-a_{i_1}}\cdots x_{i_r}^{s+1-a_{i_r}}$ is a minimal generator
of $(J^s)^{[{\bf s}]}$.  Hence, by Theorem \ref{assoc-depolar}
there exists a set $W \subseteq V_{\HH^s}$ such that $\HH^s_W$ is critically $(s+1)$-chromatic
and that ${\bf m}_W$ depolarizes to $x_{i_1}^{s+1-a_{i_1}}\cdots x_{i_r}^{s+1-a_{i_r}}$.  Now $W$
contains $s+1-a_{i_j}$ shadows of $x_{i_j}$ for each $j$.  If $x_{i_j,1}$ is not in $W$,
we can swap one of the other shadows of $x_{i_j}$ with $x_{i_j,1}$ to construct
a new set $W'$ which still has the property that $\HH^s_{W'}$ is critically $(s+1)$-chromatic
and such that $m_{W'}$ depolarizes to $x_{i_1}^{s+1-a_{i_1}}\cdots x_{i_r}^{s+1-a_{i_r}}$.
After all such swapping, the resulting set is the desired set $T$.

For the converse direction, because $\HH_{T}^s$ is critically $(s+1)$-chromatic,
we know that the depolarization of ${\bf m}_T$ is a minimal generator of $(J^s)^{[{\bf s}]}$.
But our condition on $T$ implies that this generator has the form
$x_{i_1}^{b_{i_1}}\cdots x_{i_r}^{b_{i_r}}$ with $1 \leq b_{i_1},\ldots,b_{i_r} \leq s$.
Hence $(x_{i_1}^{s+1-b_{i_1}},\ldots,x_{i_r}^{s+1-b_{i_r}})$ appears in the irredundant
irreducible decomposition of $J^s$ by Theorem \ref{correspondence},
and thus $(x_{i_1},\ldots,x_{i_r}) \in \Ass(R/J^s)$.
\end{proof}

By Theorem \ref{assoc-depolar}, it follows that critically $(d+1)$-chromatic
induced subhypergraphs correspond to an associated prime $P \in \Ass(R/J^d)
\setminus \bigcup_{e=1}^{d-1} \Ass(R/J^e)$.  That is,

\begin{corollary} \label{assprimesmingraphs}
Let $\HH = (V, \mathcal{E})$ be a simple hypergraph with cover ideal $J$. Suppose that the
induced hypergraph $\HH_P$ on $P \subseteq V$ is critically $(d+1)$-chromatic. Then
\begin{enumerate}
\item[(a)] $P \in \Ass(R/J^d)$, and
\item[(b)] $P \not\in \Ass(R/J^e)$ for any $1\leq e < d$.
\end{enumerate}
\end{corollary}

Recall that we say an ideal $I$ has the {\bf persistence property for 
associated primes} if $\Ass(R/I^s) \subseteq \Ass(R/I^{s+1})$ for all $s \geq 1$.
Computer experiments suggest that Corollary \ref{assprimesmingraphs} (a)
can be strengthened to say that $P \in \Ass(R/J^t)$ for all $t \geq d$.  In
a sequel \cite{FHVT2}, we have stated a graph theoretic
conjecture that would imply that $J(G)$ has the persistence property for
simple graphs $G$.  As further evidence, we show that
in the case of simple graphs, the prime ideals corresponding
to certain critically chromatic subgraphs persist.

We have already introduced the notion of an odd hole.  An odd
antihole is the complement of an odd hole.
A graph $G$ is a {\bf complete graph of order $n$}, denoted by $\K_n$, if $|V_G| = n$
and $\mathcal{E}_G = \{x_ix_j ~|~ 1 \leq i < j \leq n \}$.  For a subset $S \subseteq V_G$,
we call an induced subgraph $G_S$ a {\bf clique} if $G_S = \K_{|S|}$.  

\begin{corollary} \label{assprimescliques}
Let $G$ be a simple graph with cover ideal $J = J(G)$. Suppose that 
induced graph on $P \subseteq V$
is a clique, an odd hole, or an odd antihole.  Let $d+1 = \chi(G_P)$.  Then
\begin{enumerate}
\item[(a)] $P \in \Ass(R/J^t)$ for any $t \geq d$.
\item[(b)] $P \not\in \Ass(R/J^e)$ for any $1\leq e < d$.
\end{enumerate}
\end{corollary}

\begin{proof}
Localize at the clique, odd hole, or odd antihole, so $P= (x_1,\ldots,x_n)$.  
Part (b) follows immediately from 
Corollary~\ref{assprimesmingraphs}. 
If $G_P$ is a clique with $d+1 = \chi(G_P)$, then $G_P$ is a clique
of size $d+1$.  So, $P = (x_1,\ldots,x_{d+1})$.  For each $t \geq d$,
the induced graph on $\{x_{1,1},\ldots,x_{1,t-d+1},x_{2,1},\ldots,x_{{d+1},1}\}$
in $G^t$ is a clique of size $t+1$.  This clique is a critically
$(t+1)$-chromatic graph, so by Corollary \ref{classificationofass}, 
$(x_{1},\ldots,x_{d+1}) \in \Ass(R/J^t)$.

If $G_P$ is an odd hole, then $d=2$, and $G_P$ is a graph on $2m+1$
vertices.  Now use Corollary \ref{classificationofass}
and the induced subgraph of $G^t$ comprised of $t-1$ 
shadows of each of $x_1$, $x_3, \dots, x_{2m-1}$ and one shadow of each of
$x_2$, $x_4,\dots,x_{2m}$, and $x_{2m+1}$. 
If $G_P$ is an odd antihole with $\chi(G_P) = d+1$, then $G_P$
has $2d+1$ vertices.  Use the induced subgraph of $G^t$ comprised of $t+1-d$ shadows of each of $x_1$ 
and $x_2$ and one shadow of every other vertex.
\end{proof}

\begin{remark} \label{r.ann}
One can also prove Corollary~\ref{assprimescliques} without using the machinery of this section. Minimal vertex covers of cliques, odd holes, and odd antiholes have enough structure that for each $t$, one can explicitly construct monomials in $R/J^t$ for each type of graph that have the appropriate prime ideals as their annihilators.
\end{remark}

We also want to know when the sets $\Ass(R/J^s)$ stabilize. 
We give an exact answer for perfect graphs in the next section, but
Corollary \ref{assprimesmingraphs} also gives a lower bound:

\begin{corollary}\label{stablebound}
Let $\HH$ be a finite simple hypergraph with cover ideal $J$. If $a$ is the minimal
integer such that \[\bigcup_{s=1}^\infty \Ass(R/J^s) = \bigcup_{s=1}^a \Ass(R/J^s),\]
then $a \ge \chi(\HH)-1$.
\end{corollary}

This bound, however, is not optimal. We have the natural question:

\begin{question}  For each integer $n \geq 0$, does there exist a hypergraph $\HH_n$
such that the stabilization of associated primes occurs at $a \geq (\chi(\HH_n) -1) +n$?
\end{question}

There are straightforward three-chromatic examples for $n=1$ (see below) and 
$n=2$ (a graph with nine vertices, including an induced seven-cycle). We 
have also found a more complicated three-chromatic example with $n=3$ that 
involves 12 vertices and 22 edges.  However, we do not know what happens for $n \geq 4$.

\begin{example}
We illustrate some of the ideas of this section.
Consider the following graph $G$ which has $\chi(G) = 3$.
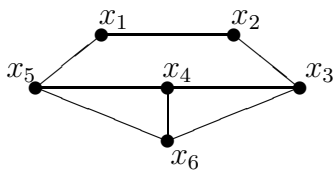
\[
\begin{picture}(110,60)(0,-10)
\put(50,0){\circle*{5}}
\put(51,-9){$x_6$}
\put(0,20){\circle*{5}}
\put(-11,24){$x_5$}
\put(50,20){\circle*{5}}
\put(48,24){$x_4$}
\put(100,20){\circle*{5}}
\put(102,24){$x_3$}
\put(75,40){\circle*{5}}
\put(74,44){$x_2$}
\put(25,40){\circle*{5}}
\put(24,44){$x_1$}
\put(50,0){\line(5,2){50}}
\put(0,20){\line(5,-2){50}}
\put(0,20){\line(1,0){100}}
\put(25,40){\line(1,0){50}}
\put(0,20){\line(5,4){25}}
\put(75,40){\line(5,-4){25}}
\put(50,0){\line(0,1){20}}
\end{picture}\]
Using \emph{Macaulay 2}, we compute $\Ass(R/J^s)$ for $s=1,2,3$;
in particular,
\[\mathfrak{m} = (x_1,\dots,x_6) \in \Ass(R/J^3) \setminus  (\Ass(R/J) \cup \Ass(R/J^2)).\]
The induced graph on $\mathfrak{m}$ is not a critically 4-chromatic graph,
so the existence of $\mathfrak{m}$ is not explained by Corollary \ref{assprimesmingraphs}.
However, by Corollary \ref{classificationofass}, there must be some set
\[\{x_{1,1},\ldots,x_{6,1}\} \subseteq T \subseteq \{x_{1,1},x_{1,2},x_{1,3},\ldots,
x_{6,1},x_{6,2},x_{6,3}\}\]
such that $G^3_T$ is a critically 4-chromatic graph.  Indeed, the required set is
\[T = \{x_{1,1},x_{2,1},x_{2,2},x_{3,1},x_{4,1},x_{5,1},x_{6,1}\}.\]
This set $T$ can be found by computing the irredundant irreducible
decomposition of $J^3$, and observing that the irreducible ideal
$(x_1^3,x_2^2,x_3^3,x_4^3,x_5^3,x_6^3)$ appears in
the decomposition.
Observe that this example also shows that the bound in 
Corollary \ref{stablebound} can be strict.  In this case $a \geq 3 > \chi(G) - 1$.
\end{example}


\section{Application: algebraic characterizations of perfect graphs} \label{s.algchar}

In this section we specialize to that case that $\HH = G$ is a finite simple
graph.  We will use the results of the previous sections to give
some new  algebraic characterizations of perfect graphs.  We begin by introducing
perfect graphs.

We denote the
size of the largest clique of $G$ by $\omega(G)$.
Clearly, $\chi(\K_n) = n$, and thus we always have $\chi(G) \geq \omega(G)$.

\begin{definition} \label{perfectdef}
A graph $G = (V_G, E_G)$ is {\bf perfect} if for every induced subgraph $G_S$, with $S \subseteq V_G$, we have
$\chi(G_S) = \omega(G_S)$.
A graph $G$ is said to be a {\bf minimal imperfect graph} if $G$ is imperfect, but for
any vertex $x \in V$, $G\setminus \{x\}$ is a perfect graph.
\end{definition}

\begin{remark}
By the Strong Perfect Graph Theorem \cite{CRST}, the only minimal imperfect graphs are the odd holes
and odd antiholes (complements of odd holes). Our aim is to give
new algebraic characterizations of perfect graphs that avoid the use
of this theorem.
\end{remark}

It is straightforward to show that a minimal imperfect graph $G$ is
 a critically $\chi(G)$-chromatic graph.  Using
Corollary \ref{assprimesmingraphs}, one can show:

\begin{lemma}\label{minimalimperfect}
Suppose that $G$ is a minimal imperfect graph on $V = \{x_1,\dots,x_n\}$ with cover ideal $J$.
Then $(x_1,\dots,x_n) \in \Ass(R/J^{\chi(G)-1})$.
\end{lemma}

Perfection is preserved when passing to the expansion:

\begin{lemma} \label{expansion-perfect}
If $G$ is a perfect graph, then $G^s$ is perfect for any integer $s \geq 1$.
\end{lemma}

\begin{proof}
By \cite[Theorem 19]{Bollobas}, if $G$ is perfect, then replacing the vertices of $G$ by
perfect graphs yields another perfect graph. To construct $G^s$ from $G$, we are replacing
each vertex of $G$ by a clique, and cliques are perfect.
\end{proof}

We next describe a large families of prime ideals $P$ that {\it cannot}
appear as associated primes by showing that the induced
graph $G_P$ has a property that prevents it from contributing
an associated prime.  Although we shall only need
a special case of this result, we prove a much more general result
since it seems of independent interest.
In what follows $N_G(x) = \{y \in V ~|~ \{x,y\} \in \E_G\}$ denotes
the {\bf set of neighbors} of $x$.

\begin{definition}
A vertex $x \in V_G$ is a {\bf simplicial vertex} of $G$ if
the induced graph on  $N_G(x)$ is a clique.  
This also implies that the induced graph on $\{x\} \cup N_G(x)$ is a clique.
\end{definition}

We need a lemma before proving our theorem about graphs with a simplicial vertex.

\begin{lemma}  \label{monomialannihilator}
Let $G$ be a finite simple graph on $V = \{x_1,\dots,x_n\}$ with cover ideal $J$.
Suppose that $(x_1,\dots,x_n) \in \Ass(R/J^d)$. If $T$ is such that $J^d:T = (x_1,\dots,x_n)$,
then $T \mid (\bm_V)^{d-1}$.
\end{lemma}

\begin{proof}
If $T \nmid (\bm_V)^{d-1}$, then there exists some $x_i$ in $T$ whose exponent is at least $d$.
Thus, the exponent of $x_i$ in $Tx_i$ is at least $d+1$.  Because $Tx_i \in J^d$, there exist
$\bm_{W_1},\dots,\bm_{W_d} \in J$ such that
$Tx_i = \bm_{W_1}\cdots \bm_{W_d}N $ for some monomial $N$. Because each $\bm_{W_j}$ is squarefree,
the exponent of $x_i$ in $\bm_{W_j}$ is at most one. So $x_i\mid N$, whence
$T = \bm_{W_1}\cdots \bm_{W_d}(\frac{N}{x_i}) \in J^d$, a contradiction since $T \not\in J^d$.
\end{proof}

\begin{theorem} \label{simplicialvertex}
Suppose that $G$ has a simplicial vertex
$x \in V$ with $N_G(x) = \{x_2,\ldots,x_r\}$.
Let $P \in \Ass(R/J^s)$.  Then
\[x \in P ~~\mbox{if and only if}~~ P = (x,x_{i_1},\ldots,x_{i_d})\]
with $\{x_{i_1},\ldots,x_{i_d}\}$ any subset of $\{x_2,\ldots,x_r\}$
of size $1 \leq d \leq \min\{r-1,s\}$.
\end{theorem}

\begin{proof}
($\Leftarrow$)  Let $\{x_{i_1},\ldots,x_{i_d}\}$ be any subset of size $1 \leq d \leq
\min\{r-1,s\}$ of $\{x_2,\ldots,x_r\}$.  Then the induced graph $G_P$ on $P = \{x,x_{i_1},\ldots,x_{i_d}\}$
is a clique of size $d+1 = \min\{r,s+1\} \leq s+1$.  Since a clique
of size $d+1$ is critically $(d+1)$-chromatic, we have $P \in \Ass(R/J^{d})$ by
Corollary  \ref{assprimesmingraphs}. But then by Corollary~\ref{assprimescliques}, 
we have $P \in \Ass(R/J^s)$ since $d \leq s$.

($\Rightarrow$)  Suppose that $x \in P \in \Ass(R/J^s)$.  We first observe that $P \neq (x)$
since this would mean that $(x)$ is a minimal associated prime of $R/J^s$,
but the minimal associated primes are the height two prime ideals
which correspond to the edges of $G$.  We use Lemma \ref{localize} to assume $G_P = G$.
There are now two cases to consider.

\noindent
{\bf Case 1.} $P = (x,x_{i_1},\ldots,x_{i_d})$ with 
$\{x_{i_1},\ldots,x_{i_d}\} \subseteq \{x_2,\ldots,x_r\}$
of size $d > \min\{r-1,s\}$.

\noindent
This case can only happen if $r-1 \geq d > s$ simply because we must
take a subset of $r-1$ elements.  In this situation, $G_P$ is a clique of size $d+1 > s+1$.
But by Corollary \ref{assprimesmingraphs}, $P \not\in \Ass(R/J^e)$ for all $e  < d$
because $G_P$ is critically $(d+1)$-chromatic.
Hence $P \not\in \Ass(R/J^s)$ since $s < d$, thus  contradicting our
assumption on $P$.

\noindent
{\bf Case 2.} $P = (x,w,\ldots)$ with $w \not\in N_G(x)\cup \{x\}$.

\noindent
Suppose that $T \in J^{s-1}\setminus J^s$ is the annihilator
such that $J^s:T = P = (x,w,\ldots)$.  There then
exist $s$ minimal vertex covers $m_1,\ldots,m_s$ (not necessarily distinct) such
that
\[Tw = m_1m_2\cdots m_sM.\]
If a minimal vertex cover does not contain $x$, it must contain
$\{x_2,\ldots,x_r\}$, or equivalently, if a generator of $J$ is not divisible
by $x$, it is divisible by $x_2x_3\cdots x_r$.
After relabeling, we may assume that
$m_1,\ldots,m_a$ are divisible by $x$ while $m_{a+1},\ldots,m_s$
are divisible by $x_2x_3\cdots x_r$ (and not by $x$).  Note that because each minimal
vertex cover of $G$ must cover the clique $\{x,x_2,\ldots,x_r\}$, we
must also have that each $m_1,\ldots,m_a$ is divisible by at least $r-2$ variables
of $x_2,\ldots,x_r$. Moreover, if there was some $m_i \in \{m_1,\ldots,m_a\}$
that was divisible by $x_2\cdots x_r$, then since $x$ also
divides $m_i$, we would have that $m_i/x \in J$, contradicting
the fact that each $m_i$ is a minimal generator of $J$. Thus, each $m_1, \dots, m_a$ is divisible by exactly $r-2$ variables of $x_2, \dots, x_r$.

Since $w \not\in N_G(x) \cup \{x\}$, the above discussion implies
that the degree of the variables $\{x_2,\ldots,x_r\}$ in $T$ must be at least
\begin{align}
(r-2)a + (r-1)(s-a) = s(r-1)-a. \label{neweq1}
\end{align} 
At the same time, we must have $x^a|T$, that is, $T = x^aT'$.

Now, because $Tx \in J^s$, there exists $s$ minimal vertex covers $m'_1, \dots, m'_s$
such that
\[Tx = m'_1m'_2\cdots m'_sM'.\]
If $x|M'$, then this would imply $T \in J^s$; hence $x \nmid M'$.
Also, $Tx = (x^aT')x = x^{a+1}T'$. Thus, $x$ appears in $a+1$ of $m'_1,\ldots,
m'_d$.  After relabeling, say $x$ divides $m'_1,\ldots,m'_{a+1}$.  But
then $m'_{a+2},\ldots,m'_s$ must be divisible by $x_2\cdots x_r$. By a similar argument as above, the degree of the variables
$\{x_2,\ldots,x_r\}$ in the monomial $m'_1\cdots m'_{a+1}m'_{a+2}\cdots m'_s$ is exactly
\begin{align}
(r-2)(a+1) + (r-1)(s-a-1) = s(r-1)-a-1. \label{neweq2}
\end{align}

It follows from (\ref{neweq1}) and (\ref{neweq2}) that some $x_i$ must divide $M'$,
i.e., $M' = x_iM''$. Now suppose that $x_i$ divided all
of $m'_1\cdots m'_{a+1}$.  Since $x_i$ also divides $m_{a+2},\ldots,m_s$,
we have $x_i^s |T$ which contradicts Lemma \ref{monomialannihilator}
which says that $T|\bm_V^{s-1}$. Hence, there exists some variable $x_i$ such that $x_i | M'$ but $x_i$ does not divide at least one of $m'_1, \dots, m'_{a+1}$. In particular, we may assume that
\[m'_{a+1} = xx_2\cdots\hat{x_i}\cdots x_rn'_{a+1}.\]
But then
\begin{eqnarray*}
Tx &= &m'_1 \cdots m'_am'_{a+1}m'_{a+2} \cdots m'_s M' \\
  & =&m'_1 \cdots m'_a(xx_2\cdots\widehat{x_i}\cdots x_rn'_{a+1})m'_{a+2} \cdots m'_s(x_iM'') \\
  & =&m'_1 \cdots m'_a(x_2\cdots x_i \cdots x_rn'_{a+1})m'_{a+2}\cdots m'_s (xM'').
\end{eqnarray*}
In the above expression, we have swapped the $x_i$ in $M'=x_iM''$ with
the $x$ in $m'_{a+1}$.  It is straightforward to see that $(x_2\cdots x_i \cdots x_rn'_{a+1})$
still corresponds to a vertex cover of $G$ and consequently,
is an element of $J$.  However, dividing by $x$ on both sides gives
\[T = m'_1 \cdots m'_a(x_2\cdots x_i \cdots x_rn'_{a+1})m'_{a+2}\cdots m'_s (M''),\]
which is an expression of $T$ written as the product of $s$ generators of $J$,
that is $T \in J^s$, thus giving the desired contradiction.
\end{proof}

We shall need the following definition for one of our characterizations of perfect graphs.

\begin{definition}
An ideal $I \subset R$ has the \textbf{saturated chain property for associated primes}
if given any associated prime $P$ of $R/I$ that is not minimal, there exists an associated prime
$Q \subsetneq P$ with height$(Q)=$height$(P)-1$.
\end{definition}

We can now prove the main result of this section.

\begin{theorem} \label{perfectgraphs}
Let $G$ be a simple graph with cover ideal $J$. Then the following are equivalent:
\begin{enumerate}
\item $G$ is perfect;
\item For all $s \ge 1$, $P=(x_{i_1}, \dots, x_{i_r}) \in \Ass(R/J^s)$ if and only if
the induced graph on $\{x_{i_1},\dots,x_{i_r}\}$ is a clique of size $1 < r \le s+1$ in $G$; 
\item For all $\chi(G) > s \ge 1$, $P=(x_{i_1}, \dots, x_{i_r}) \in \Ass(R/J^s)$ if and only if
the induced graph on $\{x_{i_1},\dots,x_{i_r}\}$ is a clique of size $1 < r \le s+1$ in $G$; and
\item For all $s \geq 1$, $J^s$ has the saturated chain property for associated primes.
\end{enumerate}
\end{theorem}

\begin{proof}
$(1) \Rightarrow (2)$
Let $G$ be a perfect graph. Consider an associated prime $P \in \Ass(R/J^s)$
for some $s$. By Theorem~\ref{assoc-depolar}, $P$ corresponds to a minimal generator $\bm_T$
with $\chi(G^s_T) > s$, that is, a critical $(s+1)$-colorable induced subgraph of $G^s$.
Since $G$ is perfect, by Lemma~\ref{expansion-perfect}, $G^s$ is also perfect. Thus,
$\omega(G^s_T) = \chi(G^s_T) > s$. This implies that there exists a subset $S \subseteq T$
such that $G^s_S$ is a clique $\K_{s+1}$. However, in this case $G^s_S$ would be a critical
$(s+1)$-colorable induced subgraph in $G^s_T$. Thus, we must have $S = T$, and $G^s_T = \K_{s+1}$
is a clique. It is now easy to see that the support of the depolarization $\overline{\bm_T}$
is also a clique of size at most $s+1$. Thus, $G_P$ is a clique.

$(2) \Rightarrow (3)$   This is immediate.

$(3) \Rightarrow (1)$  Suppose that $G$ is not perfect.  Then there
exists $P \subseteq V_G$ such that $G_P$ is a minimal imperfect graph,
with $\chi(G_P) \leq \chi(G)$.  By Lemma \ref{minimalimperfect}, we then
have $P \in \Ass(R/J^{\chi(G_P)-1})$, but $G_P$ is not a clique.

$(2) \Rightarrow (4)$.  If $s=1$, then $J$ has no embedded primes, that is, all
of its associated primes are minimal.  So the statement holds.  Now suppose
that $s \geq 2$, and that $P \in \Ass(R/J^s)$ is an embedded prime,
and thus, height$(P) = r > 2$.  By $(2)$, the induced
graph on $G_P$ is a clique of size $r \leq s+1$.  After relabeling,
we may assume that $P = (x_1,\ldots,x_r)$.  But then there
exists a clique of size $r-1$ on $\{x_1,\ldots,x_{r-1}\}$.  Now because
this clique is critically $(r-1)$-chromatic, we know
that $Q = (x_1,\ldots,x_{r-1}) \in \Ass(R/J^{r-2})$ by Corollary~\ref{assprimesmingraphs}. But then by Corollary~\ref{assprimescliques}
we have $Q \in \Ass(R/J^s)$ since $r-1 \leq s$.  Thus $J^s$
has the saturated chain property for associated primes because height$(Q)+ 1 =$ height$(P)$ and 
$Q \subsetneq P$.

$(4) \Rightarrow (1)$.  Suppose that $G$ is not perfect.  So, there exists
some subset $P \subseteq V$ such that $G_P$ is minimally imperfect.  By using
Lemma \ref{localize}, we can assume that that $P = V$.  Thus, by
Lemma \ref{minimalimperfect}, we have $(x_1,\ldots,x_n) \in \Ass(R/J^{\chi(G)-1})$.
Because $J^{\chi(G)-1}$ has the saturated chain property for associated primes,
we can find prime ideals $Q_2,Q_3,\ldots,Q_{n-1} \in \Ass(R/J^{\chi(G)-1})$ such that
\[ (x_i,x_j) = Q_2 \subset Q_3 \subset \cdots \subset Q_{n-1} \subset Q_n = (x_1,\ldots,x_n).\]
where height$(Q_i)+1 =$ height$(Q_{i+1})$ for $i=2,\ldots,n-1$.  Now
the minimal prime of height two corresponds to a clique of size two,
i.e., $G_{Q_2}$ is the clique $\K_2$. On the other hand,
$G_{Q_n} = G$ is not a clique since a clique is perfect, but $G$ is not perfect.
Thus, there exists a minimal $i$ such that $G_{Q_i}$ is a clique of size $|Q_i|$
and $G_{Q_{i+1}}$ is not a clique.  If $Q_i = (x_{i_1},\ldots,x_{i_r})$, then
$Q_{i+1} = (x_{i_1},\ldots,x_{i_r},x_{i_{r+1}})$ for some $x_{i_{r+1}}$. Since $G_{Q_{i+1}}$
is not a clique, $x_{i_{r+1}}$ is not adjacent to one of $x_{i_1},\ldots,x_{i_r}$.
Without loss of generality, say $x_{i_1}$ is not adjacent to $x_{i_{r+1}}$.

Because $Q_{i+1} \in
\Ass(R/J^{\chi(G)-1})$,
by Lemma \ref{localize}, we have that \[Q_{i+1} \in \Ass(R[Q_{i+1}]/J(G_{Q_{i+1}})^{\chi(G)-1}).\]
But in the graph $G_{Q_{i+1}}$, the vertex $x_{i_1}$ is a simplicial
vertex.  Because $x_{i_1} \in Q_{i+1}$, by Theorem \ref{simplicialvertex},
the vertex $x_{i_1}$ must be adjacent to every other vertex
$\{x_{i_2},\ldots,x_{i_r},x_{i_{r+1}}\}$.  But 
$x_{i_1}$ and $x_{i_{r+1}}$ are not adjacent.  Thus $G$ must be perfect.
\end{proof}

\begin{remark}  Using Theorem \ref{perfectgraphs} (3), we can check
if a graph is perfect in a finite number of steps.
\end{remark}

For perfect graphs, we have the persistence of associated primes
and the bound of Corollary \ref{stablebound} becomes  an equality:

\begin{corollary} \label{persist-graph}
Let $G$ be a perfect graph with cover ideal $J$.  Then
\begin{enumerate}
\item $\Ass(R/J^s) \subseteq \Ass(R/J^{s+1})$ for all integers $s \geq 1$.
\item
\[\bigcup_{s=1}^\infty \Ass(R/J^s) = \bigcup_{s=1}^{\chi(G)-1} \Ass(R/J^s).\]
\end{enumerate}
\end{corollary}

\begin{proof}  The first statement follows directly from Theorem \ref{perfectgraphs} (2).
For (2), suppose that $P \in \bigcup_{s=1}^\infty \Ass(R/J^s)$, i.e., 
$P \in \Ass(R/J^s)$ for some $s$.  Then  by Theorem \ref{perfectgraphs},
we have that $G_P$ is a clique of size at most $s+1$ in $G$.
Moreover $\chi(G_P) \leq \chi(G)$.  So, $G_P$ is a clique of size
at most $\min\{s+1,\chi(G)\}$.  By Corollary~\ref{assprimescliques},
this means that $P \in \Ass(R/J^{\chi(G)-1})$.  Thus
$\bigcup_{s=1}^\infty \Ass(R/J^s) \subseteq \bigcup_{s=1}^{\chi(G)-1} \Ass(R/J^s)$, as
desired.
\end{proof}

\begin{remark}
Theorem~\ref{perfectgraphs} enables us to describe a new infinite family 
with the saturated chain property for associated primes.
This property has been studied only in very special cases, mostly initial ideals of
$\mathcal A$-graded ideals \cite{Altmann,HostenThomas} and special
initial ideals of prime ideals \cite{Taylor}; it is of independent, purely algebraic
interest. S. Ho\c{s}ten and R. Thomas note that the condition is ``a rare property for an arbitrary monomial ideal'' \cite{HostenThomas}.
\end{remark}


\end{document}